\documentclass[11pt]{article}  

\usepackage{amsmath} 
\usepackage{amssymb}  
\usepackage{graphicx}
\usepackage{latexsym}
\usepackage{verbatim}
\usepackage{ifthen}
\usepackage{psfrag}
\usepackage{array}
\usepackage{mathtools}

\newcommand{\argmin}{\mathop{\rm arg\, min}}

\newcommand{\bdry}{\mathop{\rm bdry}\nolimits}
\newcommand{\be}{\begin{equation}}

\newcommand{\dom}{\mathop{\rm dom}\nolimits}
\newcommand{\ds}{\displaystyle} 
\newcommand{\dist}{\mathop{\rm dist}}
\newcommand{\ee}{\end{equation}}

\renewcommand{\iint}{\mathop{\rm int}\nolimits}

\newcommand{\reals}{{\mathbb R}}

\newcommand{\rrarrows}{\rightrightarrows}
\newcommand{\sign}{\mathop{\rm sign}}

\newcommand{\sol}{\phi}
\newcommand{\soldot}{\dot{\sol}}
\newcommand{\sola}{\psi}

\newcommand{\xdot}{\dot{x}}

\newtheorem{ntheorem}{\bf Theorem}[section]
\newtheorem{nlemma}[ntheorem]{\bf Lemma}
\newtheorem{ndefinition}[ntheorem]{\bf Definition}
\newtheorem{ncorollary}[ntheorem]{\bf Corollary}
\newtheorem{nproposition}[ntheorem]{\bf Proposition}
\newtheorem{nassumption}[ntheorem]{\bf Assumption}
\newtheorem{nexample}[ntheorem]{\bf Example}
\newtheorem{nremark}[ntheorem]{\bf Remark}

\newenvironment{theorem}{\begin{ntheorem}\it}{\end{ntheorem}}

\newenvironment{corollary}{\begin{ncorollary}\it}{\end{ncorollary}}
\newenvironment{proposition}{\begin{nproposition}\it}{\end{nproposition}}
\newenvironment{assumption}{\begin{nassumption}\it}{\end{nassumption}}

\newenvironment{example}
{\begin{nexample}\rm}{\hspace*{\fill}$\bigtriangleup$\end{nexample}}

\newcommand{\eop}
           {\hspace*{\fill}{$\vcenter{\hrule height1pt 
                     \hbox{\vrule width1pt height5pt 
            \kern5pt \vrule width1pt} \hrule height1pt}$} }

\newenvironment{proof}
{\par\noindent\textbf{Proof.}}{\eop\smallskip\vskip 3 pt}

\usepackage{color}

\begin{document}


\title{A unifying convex analysis and switching system approach to consensus with undirected communication graphs}       
\author{Rafal Goebel\thanks{Department of Mathematics and Statistics, Loyola University Chicago, 
1032 W. Sheridan Road, Chicago, IL 60660. Email: {\it rgoebel1@luc.edu}. 
This work was partially supported by the Simons Foundation Grant 315326.} 
and Ricardo Sanfelice\thanks{Department of Computer Engineering,
University of California, Santa Cruz, CA, 95064, USA. E-mail:
{\it ricardo@ucsc.edu}. This work was partially
supported by the National Science Foundation CAREER Grant 
ECS-1150306 and by the Air Force Office of Scientific Research YIP
Grant FA9550-12-1-0366.}
}

\maketitle

\begin{abstract}
Switching between finitely many continuous-time autonomous steepest descent dynamics for convex functions is considered. 
Convergence of complete solutions to common minimizers of the convex functions, if such minimizers exist, is shown. 
The convex functions need not be smooth and may be subject to constraints. 
Since the common minimizers may represent consensus in a multi-agent system modeled by an undirected communication graph, 
several known results about asymptotic consensus are deduced as special cases. 
Extensions to time-varying convex functions and to dynamics given by set-valued mappings more general
than subdifferentials of convex functions are included. 
\end{abstract}

\date{\today}


\section{Introduction} 

This technical note presents a convex analysis and switching systems-based approach to proving convergence
to consensus in multi-agent systems or networks modeled by undirected graphs. 
The approach unifies and generalizes results in, for example, 
\cite{OhParkAhn15Automatica},
\cite{OlfatiSaberFaxMurray07IEEE}, 
\cite{ShiJohanssonHong13TAC}, 
\cite{QiuLiuXie16Automatica}, 
\cite{ShiProutiereJohansson15SICON},
\cite{YangMengShiHongJohansson16TAC}, 
and more, while making weaker assumptions. 
Even for the linear case, the approach uses no linear or spectral analysis. 
Certain nonlinearities, given by gradients or subdifferentials or convex functions,
convex constraints, and gradient projections fit in the approach. 

The convergence result is given for a system that switches between finitely many 
maximal monotone mappings, which are subdifferentials of convex functions restricted
to convex sets. For dynamics given by a single, but multivalued, maximal monotone mapping, 
for example for the steepest descent for a convex function, 
existence and uniqueness of solutions, their nonexpansive property, and convergence to equilibria, or minimizers in the
steepest descent case, is well studied; 
see \cite{Brezis} for a classical reference and \cite{PeypouquetSorin20JCA} for a recent survey. 
The case of a single time-varying maximal monotone mapping has seen less treatment; 
see \cite{AttouchCabotCzarnecki18TAMS}. 
For finitely many convex functions, for continuous-time dynamics driven by the minimum norm velocity 
that can be generated via arbitrary switching between steepest descents, 
convergence to Pareto-optimal points is expected;
see \cite{AttouchGarrigosGoudou15JMAA} and the references therein.  
(Common minimizers, if they exist, are Pareto-optimal.)
Stability of such points holds \cite{Miglierina04SVA}, for differential inclusions driven by
the so-called pseudogradient directions, that include the mentioned minimum norm velocity
and are related to descent directions for discrete-time multiobjective optimization \cite{FliegeSvaiter00MMOR}. 
These cases, in general, do not apply to switching dynamics. 
Results on convergence to common minimizers
for switching between finitely many convex functions does not appear to have been written down,
though similar ideas are, of course, present in alternating projections or alternating proximal point
optimization algorithms in discrete-time.

Some convex-analytic methods have been used, to an extent, in
the continuous-time consensus setting
\cite{ShiJohanssonHong13TAC}, 
\cite{ShiProutiereJohansson15SICON},
\cite{YangMengShiHongJohansson16TAC},
\cite{ZengYiHong17TAC},
as well as for the discrete-time case in \cite{NedicOzdaglarParrilo10TAC}, \cite{ZhuMartinez12TAC}, etc. 
The setting of this note encompasses, in the undirected graph case, the mentioned continuous-time works
while making weaker assumptions. The connection between consensus, gradient flow for a convex
function, and convergence of solutions to a minimizer of the function has been made before, 
in the linear case \cite{OlfatiSaberFaxMurray07IEEE} and beyond \cite{ShiProutiereJohansson15SICON}.  
This note carries the idea further. Preliminary work is in \cite{GoebelSanfelice18CDC}.
Finally, monotonicity as mentioned above is different from what is considered
in monotone systems; see \cite{ManfrediAngeli17Automatica} for consensus results in monotone systems.

\section{Main result} 

Consider the switching system
\be 
\label{ss}
\xdot\in -M_q(x)
\ee 
where the data is subject to the following assumption,
the background for which is presented in Section \ref{convex background section}. 

\begin{assumption}
\label{data assumption}
$Q=\{1,2,\dots,p\}$, and for every $q\in Q$, $M_q:\reals^n\rrarrows\reals^n$ 
is a set-valued mapping given by 
$$
M_q(x)=\partial f_q(x)+N_{C_q}(x),
$$
where $f_q:\reals^n\to\reals$ is a convex function;
$\partial f_q(x)$ is its subdifferential at $x$, in the sense of convex analysis;
$C_q\subset\reals^n$ is a nonempty closed convex set; 
and $N_{C_q}(x)$ is the normal cone to $C_q$ at $x$. 
\end{assumption}

A {\em switching signal} is a function $\sigma:[0,\infty)\to Q$ such that there exists a sequence $0=t_0<t_1<t_2<\dots$ such
that $\sigma$ is constant on $[t_j,t_{j+1})$ for $j=0,1,2,\dots$. 
Given such a signal $\sigma$, a {\em solution} to \eqref{ss} is a locally absolutely continuous function
$\sol:\dom\sol\to\reals^n$, where $\dom\sol$ is $[0,T)$, $[0,T]$, or $[0,\infty)$, such that
$$
\soldot(t)\in M_{\sigma(t)}(\sol(t)) \quad \mbox{for almost every}\ t\in\dom\sol.
$$
A solution $\phi$ is {\em maximal} if it cannot be
extended and {\em complete} if $\dom\sol=[0,\infty)$. Since $N_{C_q}(x)=\emptyset$ if 
$x\not\in C_q$, \eqref{ss} requires $\phi(t)\in C_{\sigma(t)}$
for every $t \in \dom \phi$, maximal solutions $\phi$ may fail to be complete. 

Of interest is the behavior of complete solutions to \eqref{ss} when the functions $f_q$ have common minimizers over $C_q$. 

\begin{assumption}\
\label{A assumption} 
The set $A$ is nonempty, where
$$A:=\bigcap_{q\in Q} A_q \qquad \mbox{and} \qquad A_q:=\argmin_{x\in C_q} f_q(x).$$
\end{assumption} 

For every common minimizer $a\in A$, the function 
$V(x):=\frac{1}{2}\|x-a\|^2$  
satisfies 
\be 
\label{V dec}
\frac{d}{dt}V(\sol(t))\leq \min_{x\in C_{\sigma(t)}} f_{\sigma(t)}(x)-f_{\sigma(t)}(\sol(t)) \leq 0
\ee 
along every solution $\phi$ to \eqref{ss},
which follows from the definition of the convex subdifferential \eqref{convex subdifferential}.
In particular, every $a\in A$ is Lyapunov stable for \eqref{ss}.
The inequality \eqref{V dec} can be viewed as
$\frac{d}{dt}V(\sol(t))\leq-W_{\sigma(t)}(\sol(t))$, where $W_q(x)=f_q(x)-\argmin_{x\in C_q}f_q(x)$,
and it is natural to expect some, or all, $W_q$ to asymptotically approach $0$ along complete solutions.

\begin{assumption}
\label{sigma assumption}
The switching signal $\sigma$ is such that, for each $q\in Q$,
$$
\mu(T_q(\sigma))=\infty \quad \mbox{where}\  T_q(\sigma):=\{t\in[0,\infty)\, |\, \sigma(t)=q\}.
$$
\end{assumption} 

Above, $\mu$ is the Lebesgue measure but reduces to the sum of lengths of intervals, since $T_q(\sigma)$ is a union of intervals. 
The assumptions holds, for example, if the switching signal has a positive dwell time $\tau_D>0$, i.e., if discontinuities
of $\sigma$ are separated by at least $\tau_D$, and there exists $T>0$ such that, for every $t\geq 0$, the range of $\sigma$
over $[t,t+T]$ is $Q$. But the assumption is more general. 

Under the stated assumptions, the main result is that every complete solution to \eqref{ss} converges to a common minimizer in $A$.
 
\begin{theorem} 
\label{main theorem}
Under assumptions \ref{data assumption}, \ref{A assumption}, and \ref{sigma assumption}
every complete solution to \eqref{ss} is such that $\lim_{t\to\infty} x(t)$ exists and belongs to $A$. 
\end{theorem} 

In the setting of Theorem \ref{main theorem}, since every $a\in A$ is Lyapunov stable,
the set $A$ has the property called pointwise asymptotic stability, also known as semistability. See 
\cite{GoebelSanfelice18SICON} and the references therein.

Of particular interest, motivated by consensus questions, is when and how the set of common minimizers $A$ 
is related to the consensus subspace. Let $n=km$, where $k$ represents the number of $m$-dimensional
agents. For convenience, $x\in\reals^n$ is $(x_1,x_2,\dots,x_k)$, with $x_i\in\reals^m$, and the
{\em consensus subspace} is 
$$
CS:=\{x\in\reals^n\, |\, x_1=x_2=\dots=x_k\}
$$
If complete solutions to \eqref{ss} are such that their limits exist and are in $CS$, it is said that {\em the agents
reach consensus}. 

\begin{example} 
\label{linear example} 
Let $f:\reals^n\to\reals$ be 
$$
f(x)=\frac{1}{4} \sum_{i,j=1}^k a_{ij}\|x_i-x_j\|_p^p,
$$
where, for $i,j=1,\dots,k$, $a_{ij}=a_{ji}\geq 0$, $p\in[1,\infty)$, and $\|\cdot\|_p$ is the usual $p$-norm. 
Let $C=\reals^n$. 
Then $f$ is a convex function. For $p>1$, 
$\partial f$ reduces to $\nabla f$, $N_C(x)=\{0\}$ for all $x\in\reals^n$, $M$ reduces to $\nabla f$, and
\eqref{ss} becomes $\xdot=-\nabla f(x)$. 
For $p=2$, more explicitly, this differential equation is, for $i=1,\dots,k$,
\be 
\label{linear consensus}
\xdot_i=\frac{1}{2}\sum_{j=1}^k a_{ij}\left( x_j-x_i \right).
\ee 
For $p=1$, \eqref{ss} turns to 
$$\xdot_i=\sum_{j=1}^k a_{ij}\sign\left( x_j-x_i \right), \quad i=1,\dots,k,$$
where $\sign$ is taken coordinate-wise and, by convention, $\sign(0)=0$. 
(The convex subdifferential of the absolute value is set valued at $0$, and equals $[-1,1]$, 
but this does not affect the solutions.) Second-order controllers with the sign function are, for example, 
in \cite{JafarianDePersis13ACC}. 

Here, $A=\argmin f$, and clearly $CS\subset A$. If the communication graph between agents is connected,
and $a_{ij}=a_{ji}>0$ for every undirected edge between $i$-th and $j$-th agent, then $A\subset CS$ and
thus $A=CS$. Indeed, in such a case, for any two agents, say $i$-th and $j$-th, there exist agents $i=i_0$, $i_1$, $i_2$, ... ,
$i_L=j$ such that there is an edge between $i_{l-1}$ and $i_l$, and so $a_{i_{l-1},i_l}=a_{i_l,i_{l-1}}>0$, for each $l=1,2,\dots,L$; 
$f(x)=\min f=0$ is possible only if $a_{i-1,i}(x_{i-1}-x_i)^2=0$ for each $l=1,2,\dots,L$, and
so $x_{i-1}=x_i$ for each $l=1,2,\dots,L$, and consequently, $x_i=x_j$. 
Theorem \ref{main theorem} implies that, in this case, agents reach consensus. 
This conclusion is, of course, well-known; see, for example, \cite[Lemma 1]{OlfatiSaberFaxMurray07IEEE}. 

On the other hand, if $CS\not=A$, which requires that the communication graph be not connected,
for some initial conditions (in particular, those in $A\setminus CS$, which are equilibria of \eqref{linear consensus}),
the agents do not reach consensus. 
\end{example} 

\begin{example} 
\label{linear infty example} 
Let $f:\reals^n\to\reals$ be 
$$
f(x)=\frac{1}{4} \sum_{i,j=1}^k a_{ij}\|x_i-x_j\|^2_\infty,
$$
where, for $i,j=1,\dots,k$, $a_{ij}=a_{ji}\geq 0$, and $\|u\|_\infty=\max_{i=1,2,\dots,m}|u_i|$. 
Let $C=\reals^n$. 
Then $f$ is a convex function and $\xdot\in-\partial f(x)$ becomes
$$
\xdot_i=\sum_{j=1}^k a_{ij} \left( x_{j,l(i,j)}-x_{i,l(i,j)} \right), 
$$
when $x$ is such that, for every $i,j=1,2,\dots,k$, there exists a unique $l(i,j)\in\{1,2,\dots,m\}$ such that
$\|x_i-x_j\|_\infty=|x_{i,l(i,j)}-x_{j,l(i,j)}|$. Then, the contribution of the $j$-th agent to the velocity 
of the $i$-th one is only in the coordinate in which their positions differ the most. 
At other points $x$, $f$ is not differentiable and $\xdot$ can be found as the minimum norm element
of $-\partial f(x)$.

If the norm is not squared, in the formula for $f$, then the dynamics, at $x$ as specified above, is
$$
\xdot_i=\frac{1}{2}\sum_{j=1}^k a_{ij}  \sign \left( x_{j,l(i,j)}-x_{i,l(i,j)} \right). 
$$
\end{example} 

\begin{example} 
\label{linear q example} 
Let $f_q:\reals^n\to\reals$ be given by 
$$
f_q(x)=\frac{1}{4} \sum_{i,j=1}^k a_{ij}(q)(x_i-x_j)^2,
$$
where, for $i,j=1,\dots,k$, $a_{ij}(q)=a_{ji}(q)\geq 0$. Let $C_q=\reals^n$. Then, as suggested by 
Example \ref{linear example}, the switching system \eqref{ss} is
\be 
\label{linear q consensus}
\xdot_i=\sum_{j=1}^k a_{ij}(q)\left( x_j-x_i \right).
\ee 
Here, $A=\cap_{q\in Q} \argmin f_q$. Clearly, $CS\subset A_q$ for each $q$ and thus $CS\subset A$.
If 
\begin{itemize} 
  \item[(CG)] 
the graph given by the union of the communication graphs associated to each $q \in Q$ is connected, 
and, for each $q \in Q$, $a_{ij}(q)=a_{ji}(q)>0$ for every undirected edge between $i$-th and $j$-th agent in the $q$-th graph,

\end{itemize} 
then $A\subset CS$ and thus $A=CS$. 
Under Assumption \ref{sigma assumption}, Theorem \ref{main theorem} implies that, in this case, agents reach consensus. 
\end{example} 

\begin{example} 
\label{linear q local example} 
Let $f_q:\reals^n\to\reals$ be 
$$
f_q(x)=\sum_{i=1}^k g_i(x_i) + \frac{1}{4} \sum_{i,j=1}^k a_{ij}(q)(x_i-x_j)^2,
$$
where $g_i:\reals^n\to\reals$ are differentiable convex functions 
and $a_{ij}(q)=a_{ji}(q)\geq 0$. Let $C_q=\reals^n$. 
Then, the switching system \eqref{ss} is
\be 
\label{local linear q consensus}
\xdot_i=-\nabla g_i(x_i)+\sum_{j=1}^k a_{ij}(q)\left( x_j-x_i \right),
\ee 
and $-\nabla g_i(x_i)$ can be considered as local dynamics of the $i$-th agent. 
If $D:=\cap_{i=1}^k D_i\not=\emptyset$, where $D_i:=\argmin g_i$, then for every
$\xi\in D$, $(\xi,\xi,\dots,\xi)\in A$. If, additionally, (CG) holds, then every
$a\in A$ has the form $(\xi,\xi,\dots,\xi)$ with $\xi\in D$. 
Thus, under both conditions, agents reach consensus while also finding a common
minimizer to $g_i$. This recovers \cite[Theorem 4.1]{ShiProutiereJohansson15SICON};
in fact generalizes it to the switching setting.

Furthermore, if $g_i(x_i)=\frac{1}{2}\dist_{D_i}(x_i)^2$, where 
$D_i\subset\reals^m$ is a nonempty closed and convex set, and $\dist_{D_i}(x_i)$ 
is the distance of $x_i$ from $D_i$, then
$\nabla g_i(x_i)=x_i-P_{D_i}(x_i)$, where $P_{D_i}(x_i)$ is the projection of $x_i$ onto
$D_i$, i.e., the point in $D_i$ closest to $x_i$, and \eqref{local linear q consensus} becomes
$$
\xdot_i=P_{D_i}(x_i)-x_i+\sum_{j=1}^k a_{ij}(q)\left( x_j-x_i \right).
$$
If $D:=\cap_{i=1}^k D_i\not=\emptyset$ and (CG) holds, every $a\in A$ has 
the form $(\xi,\xi,\dots,\xi)$ with $\xi\in D$, and agents
reach consensus while finding a point in this intersection. This is an autonomous
version of a result in \cite{ShiJohanssonHong13TAC}; the time-varying case
can be addressed as in Section \ref{time varying case section}. 
\end{example}

\begin{example}
\label{linear q projection example 0}
Let $f_q$ be as in Example \ref{linear q example} and $C$ be given by
$$
C=C_1\times C_2\times\dots\times C_k
$$
for nonempty closed convex $C_i\subset\reals^m$. 
Then, the switching system \eqref{ss} has the same solutions as the system
$$
\xdot_i=P_{T_{C_i}(x_i)}\left( \sum_{j=1}^k a_{ij}(q)\left( x_j-x_i \right) \right),
$$
where $P_{T_{C_i}(x_i)}(v)$ is the projection of $v$ onto the space tangent to $C_i$ at $x_i$;
see Section \ref{convex background section} for justification.   
Thus, \eqref{ss} reduces to \eqref{linear q consensus} projected onto $C_i$. Note that
the projection can be done locally, i.e., by each agent independently.   
This is the continuous-time version of the dynamics considered, for example, in \cite{NedicOzdaglarParrilo10TAC}.
\end{example}

\section{Proof of the main result}

\subsection{Auxiliary convergence result}

\begin{theorem} 
\label{uber theorem}
Let $Q=\{1,2,\dots,p\}$. 
Let $\sigma$ be a switching signal that satisfies Assumption \ref{sigma assumption} 
and $\sol:[0,\infty)\to\reals^n$ be a bounded, uniformly continuous, and locally absolutely continuous function. 
Suppose that there exist
\begin{itemize}
  \item a differentiable function $V:\reals^n\to\reals$ which is bounded below;
  \item for every $q\in Q$, a lower semicontinuous function $W_q:\reals^n\to[0,\infty)$;
\end{itemize} 
such that, 
\begin{itemize} 
  \item[(a)] 
for almost all $t\in[0,\infty)$, 
\be 
\label{V decrease}
\frac{d}{dt}V(\sol(t))\leq-W_{\sigma(t)}(\sol(t))
\ee
  \item[(b)] 
$\ds{A:=\bigcap_{q\in Q} A_q\not=\emptyset}$ where $\ds{A_q:=\argmin W_q}$;
  \item[(c)] 
for every $q\in Q$, on every interval on which $\sigma(t)=q$, for every $a_q\in A_q$,
$t\mapsto\|\sol(t)-a_q\|$ is nonincreasing.
\end{itemize} 
Then $\sol(t)$ converges, as $t\to\infty$ to a point in $A$. 
\end{theorem} 

The result above borrows some ideas from results involving multiple Lyapunov functions, for example 
\cite[Theorem 7]{HespanhaLiberzonAngeliSontag05TAC}, and possibly other related results \cite{DanieL},
but $V$ and $W_q$ are not related through observability or detectability-like conditions,
and assumption (c) is quite specific for the subdifferential flow for a convex function and
unusual for the multiple Lyapunov functions results. Without this assumption, the result fails.
Indeed, consider $\xdot=y$, $\dot{y}=-x$, $V(x,y)=x^2+y^2$,
and $W_q$, for $q=1,2,3,4$ being the distance from the $q$-th quadrant. For any complete (and periodic!) solution $\sol$ not from the origin, 
with $\sigma(t)=q$ if $\sol(t)$ is in the $q$-th quadrant, all assumptions except (d) are satisfied, 
in particular \eqref{V decrease} holds with both sides equal to $0$. Note that also, when $\sigma(t)=q$, then the distance of $\sol(t)$
from the $q$-th quadrant is $0$ and thus nondecreasing. But the conclusion fails: $A$ is the origin.

\

\begin{proof} 
Let the sets $T_q(\sigma)$ come from Assumption \ref{sigma assumption}.
For each $q\in Q$, there exists a sequence $t_{q,j}\to\infty$ as $j\to\infty$ such that $t_{q,j}\in T_q(\sigma)$ and 
$W_q(\sol(t_{q,j}))\to 0$. 
Indeed, in the opposite case there exists $\delta>0$ such that, for all large enough $t\in T_q(\sigma)$,
$W_q(\sol(t))\geq\delta$ and thus $\frac{d}{dt}V(\sol(t))\leq-\delta$. 
Since $\mu(T_q(\sigma))=\infty$, this contradicts $V$ being bounded below.  
Since $\sol$ is bounded, without loss of generality one can assume that $\sol(t_{q,j})$ converge, 
as $j\to\infty$, to a limit denoted $a_q$. By lower semicontinuity and nonnegativity of $W_q$, $W(a_q)=0$ and $a_q\in A_q$. 
 
Let $Q_1=\{q\in Q\, |\, a_1\in A_q\}$. It will be shown that $Q_1=Q$. 
Suppose that $Q_1\not=Q$ and, without loss of generality, that $2\not\in Q_1$. 
Let $d>0$ be such that 
$$d<\min_{q\not\in Q_1} \min_{x\in A_q} \|x-a_1\|,$$
and define  
$$
\delta=\min_{q\not\in Q_1} \min_{x:\|x-a_1\|\leq d} W_q(x),
$$
which is positive since $W_q$ are continuous and if $\|x-a_1\|\leq d$ then $x\not\in A_q$ for $q\not\in Q_1$. 
Recall that $\sol(t_{1,j})\to a_1$ and $\sol(t_{2,j})\to a_2$ as $j\to\infty$. 
Thus, there exist sequences $\tau_j$, $\tau'_j$ with $\tau_j<\tau'_j<\tau_{j+1}$ 
such that, for each $j=1,2,\dots$, 
$$\|\sol(\tau_j)-a_1\|=d/2, \quad \|\sol(\tau'_j)-a_1\|=d,$$
$$d/2\leq\|\sol(t)-a_1\|\leq d\ \forall t\in[\tau_j,\tau'_j].$$
For almost all $t\in[\tau_j,\tau_j']$ such that $\sigma(t)\in Q_1$, i.e., such that $a_1\in A_{\sigma(t)}$,
$\|\sol(t)-a_1\|$ is nonincreasing. 
Let $S_j:=\{t\in[\tau_j,\tau'_j]\, |\, \sigma(t)\not\in Q_1\}$.
Since $\|\sol(\tau'_j)-\sol(\tau_j)\|\geq d/2$ and $\sol$ is uniformly continuous, there exists $\eta>0$ such that,
for each $j$, $\mu(S_j)\geq\eta$.
But for almost all $t\in S:=\bigcup_{j=1}^\infty S_j$, 
$\frac{d}{dt}V(\sol(t))\leq-\delta$. 
Since $\mu(S)=\infty$, this is a contradiction!
Consequently, $Q_1=Q$ and $a_1\in A$.
Assumption (c) ensures now that $\|\sol(t)-a_1\|$ is nonincreasing and hence $\sol(t)$ converges to $a_1$. 
\end{proof}

\subsection{Convex analysis background}
\label{convex background section} 

For details on the convex analysis material collected below, see \cite{VA}.
For details on differential inclusions, including \eqref{subdifferential inclusion}, see
\cite{Brezis71}, \cite{AubinCellina}. 

Let $g:\reals^n\to\reals\cup\{\infty\}$ be a proper (i.e., finite somewhere), lower semicontinuous (lsc), 
and convex function. Let $\dom g$ be the {\em effective domain} of $g$, i.e. the set $\{x\in\reals^n\, |\, g(x)<\infty\}$.
The {\em convex subdifferential} mapping of $g$ is the set-valued mapping $\partial g:\reals^n\rightrightarrows\reals^n$,
with $\partial g(x)$ given by
\be 
\label{convex subdifferential}
\left\{y\in\reals^n\, |\, g(x')\geq g(x)+y\cdot(x'-x)\ \forall x'\in\reals^n\right\}.
\ee 
The subdifferential mapping $\partial g$ is maximal monotone, and consequently, for the
differential inclusion
\be 
\label{subdifferential inclusion}
\xdot\in-\partial g(x)
\ee 
one has
\begin{itemize} 
\item[(a)] 
For every $x_0$ in the closure of $\dom g$ 
there exists a unique maximal solution to \eqref{subdifferential inclusion} 
with $\sol(0)=x_0$ and this solution is complete.
\item[(b)] 
For any two complete solutions $\sol$, $\sola$ to \eqref{subdifferential inclusion}, $t\mapsto\|\sol(t)-\sola(t)\|$ 
is nonincreasing.  In particular, the solutions to \eqref{subdifferential inclusion} depend continuously on initial conditions, in the uniform norm over $[0,\infty)$. 
 \item[(c)] 
For every solution $\sol$ to \eqref{subdifferential inclusion} and for almost all $t\geq 0$ 
$$\xdot(t)=m\left(-\partial g(x(t))\right),$$
where $m(S)$ is the element of the closed set $S$ with minimum norm. 
\end{itemize}

Now let $g:\reals^n\to\reals$ be convex and let $D\subset\reals$ be a nonempty, closed, and convex set.  
For such $g$, $\partial g$ is locally bounded. 
Define $g_D:\reals^n\to\reals\cup\{\infty\}$ by
\be 
\label{gD}
g^D(x)=\left\{\begin{matrix} g(x) & \mbox{if} & x\in D \cr \infty & \mbox{if} & x\not\in D\end{matrix}\right.
\ee
Then $g^D$ is proper, lower semicontinuous, and convex, and its subdifferential is given by
$$
\partial g^D(x)=\left\{\begin{matrix} \partial g(x) & \mbox{if} & x\in\iint D \cr 
\partial g(x)+N_D(x) & \mbox{if} & x\in\bdry D \cr
\emptyset & \mbox{if} & x\not\in D\end{matrix}\right.
$$
Above, $\iint D$ and $\bdry D$ stand for the interior and the boundary of $D$, and
$N_D(x)$ is the {\em normal cone} to $D$ at $x$, given by 
$$N_D(x)=\{v\in\reals^n\, |\, v\cdot(x'-x)\leq 0\ \forall x'\in D\}.$$
Since $0\in N_D(x)$ for every $x\in D$, the minimum norm element of $-\partial g(x)-N_D(x)$ is locally bounded. 
The general facts above can be applied to $M_q=\partial f_q+N_{C_q}$.

\begin{proposition} 
\label{existence etc}
Under Assumption \ref{data assumption},
given a switching signal $\sigma:[0,\infty)\to Q$, 
and $x_0\in C_{\sigma(0)}$, 
there exists a unique maximal solution $\sol:\dom\sol\to\reals^n$ to \eqref{ss} with $\sol(0)=x_0$.
If this solution is bounded, it is uniformly continuous.  
This solution is complete if, for example, $C_1=C_2=\dots=C_p$. 
\end{proposition}

If $g:\reals^n\to\reals$ is differentiable, then for $x\in\bdry D$, the minimum norm element of $-\nabla g(x)-N_D(x)$
is the same as the projection of $-\nabla g(x)$ onto the tangent cone to $D$ at $x$:
$$-m\left(\nabla g(x)+N_D(x)\right)=P_{T_D(x)}\left(-\nabla g(x)\right),$$
see \cite{Henry73JMAA} or the recent \cite[Corollary 2]{BrogliatoDaniilidisLemarechalAcary06SCL}.
Thus, solutions to $\xdot\in-\partial g^D(x)$ are the same as solutions to projected gradient
dynamics $\xdot=P_{T_D(x)}\left(-\nabla g(x)\right)$. 
Above, the {\em tangent cone} to $D$ at $x$ is
$$T_D(x)=\{u\in\reals^n\, |\, u\cdot v\leq 0\ \forall v\in N_D(x)\}$$

\subsection{Proof of Theorem \ref{main theorem}}

Take $a\in A$. Let $V(x)=\frac{1}{2}\|x-a\|^2$. 
Let $\sigma$ be a switching signal and $\sol:\dom\sol\to\reals^n$ a solution to \eqref{ss}. 
Then, since $M_q=\partial f_q^{C_q}$ where $f_q^{C_q}$ is constructed from $f_q$ and $C_q$ via \eqref{gD},
$$
\frac{d}{dt}V(\sol(t))=(\sol(t)-a)\cdot\soldot(t)\leq f^{C_{\sigma(t)}}_{\sigma(t)}(a)-f^{C_{\sigma(t)}}_{\sigma(t)}(\sol(t))\leq 0
$$
where the first inequality above comes directly from the definition of the subdifferential \eqref{convex subdifferential}.
This confirms Lyapunov stability of $a$ for \eqref{ss} and ensures that $\sol$ is bounded.
Let $\sol$ be complete. 
For $q\in Q$, let $W_q:\reals^n\to[0,\infty)$ be given by $W_q(x)=f^{C_q}_q(x)-\min_{x\in C_q} f_q(x)$. Note that
$W_q(x)=f_q(x)-\min_{x\in C_q} f_q(x)$ for $x\in C_q$, $W_q(x)=\infty$ otherwise, and so $W_q$ is lsc. 
Then $\sol$ satisfies \eqref{V decrease} and Theorem \ref{uber theorem} implies the result.

\section{Extensions} 

Without Assumption \ref{sigma assumption}, one can still conclude convergence of solutions
to \eqref{ss}, with limits in a set that depends on the switching signal. 
Given a switching signal $\sigma$, let 
$$
Q_\infty(\sigma):=\{q\in Q\, |\, T_q(\sigma)=\infty\}, \quad
A_\infty(\sigma):=\bigcap_{q\in Q_\infty} A_q.
$$

\begin{corollary} 
\label{main infty theorem}
Under assumptions \ref{data assumption} and \ref{A assumption}, given a switching signal $\sigma$,
every complete solution to \eqref{ss} is such that $\lim_{t\to\infty} x(t)$ exists and belongs to $A_\infty(\sigma)$. 
\end{corollary} 

This can be deduced from the proof of Theorem \ref{main theorem}.
For consensus purposes, the question then becomes whether $A_\infty\subset CS$. 

Two other extensions of Theorem \ref{main theorem} conclude this technical note.

\subsection{Switching and mildly time-varying case} 
\label{time varying case section} 

Consider the switching system
\be 
\label{sst}
\xdot\in -M_q(x,t).
\ee 

\begin{assumption}
\label{time varying assumption}\ 
$Q=\{1,2,\dots,p\}$, and for every $q\in Q$, $M_q:\reals^n\times[0,\infty)\rrarrows\reals^n$ 
is a set-valued mapping given by 
$$
M_q(x,t)=\partial f_q(x,t)+N_{C_q}(x),
$$
where $f_q:\reals^n\times[0,\infty)\to\reals$ is a function convex in $x$ for each fixed $t$;
$\partial f_q(x,t)$ is the subdifferential, in the sense of convex analysis, of $x\mapsto f_q(x,t)$ at $x$;
$C_q\subset\reals$ is a nonempty closed convex set; and $N_{C_q}(x)$ is the normal cone to $C_q$ at $x$. 
Additionally 
\begin{itemize} 
  \item[(a)] 
for each $q\in Q$, there exists a closed convex set $A_q\subset\reals^n$ such that, for every $t\in[0,\infty)$,  
$$\argmin_{x\in C_q} f_q(x,t)=A_q;$$
  \item[(b)] the set $A:=\bigcap_{q\in Q} A_q$ is nonempty;
  \item[(c)] 
for each $q\in Q$, there exists a convex function $g_q:\reals^n\to\reals$ such that
\begin{itemize} 
  \item $\argmin_{x\in C_q} g_q(x)=A_q$;
  \item for every $x\in\reals^n$, $t\in[0,\infty)$, $g_q(x)\leq f_q(x,t)$;
\end{itemize} 
  \item[(d)]
for each $q\in Q$, $\partial f_q(x,t)$ are locally bounded in $x$ uniformly in $t\in[0,\infty)$. 
\end{itemize} 
\end{assumption} 

\begin{theorem} 
\label{time varying theorem}
Under assumptions \ref{time varying assumption} and \ref{sigma assumption},
every complete solution to \eqref{sst} is such that $\lim_{t\to\infty} x(t)$ exists and belongs to $A$. 
\end{theorem} 

\begin{proof} 
The proof relies on Theorem \ref{uber theorem} and is identical to that of Theorem \ref{main theorem}, 
after noting that the same $V$ leads to
\begin{eqnarray*}
\frac{d}{dt}V(\sol(t))
&\leq& 
f^{C_{\sigma(t)}}_{\sigma(t)}(a,t)-f^{C_{\sigma(t)}}_{\sigma(t)}(\sol(t),t) \phantom{\int} \cr
&\leq& 
g^{C_{\sigma(t)}}_{\sigma(t)}(a)-g^{C_{\sigma(t)}}_{\sigma(t)}(\sol(t)) \leq 0,
\end{eqnarray*}
one can take $W_q:\reals^n\to[0,\infty)$ given by $W_q(x)=g^{C_q}_q(x)-\min_{x\in C_q} g_q(x)$,
and uniform continuity of every bounded solution comes from Assumption \ref{time varying assumption} (d). 
\end{proof} 

\begin{example} 
\label{linear q t example} 
Let $f_q:\reals^n\times[0,\infty)\to\reals$ be given by 
$$
f_q(x,t)=\frac{1}{4} \sum_{i,j=1}^k a_{ij}(q,t)(x_i-x_j)^2,
$$
where, for $i,j=1,\dots,k$, $a_{ij}(q)=a_{ji}(q)\geq 0$. Let $C_q=\reals^n$. Then \eqref{sst} is, for $i=1,\dots,k$,
$$
\xdot_i=\sum_{j=1}^k a_{ij}(q,t)\left( x_j-x_i \right).
$$
Suppose that there exist $a_*,a^*>0$ such that, for every $q$, every $i,j=1,2,\dots,k$,
either $a_{ij}(q,t)=0$ for all $t\geq 0$ or 
$$
a_*\leq a_{ij}(q,t)\leq a^* \quad \mbox{for all}\ t\geq 0. 
$$
Assumption like the above is made, for example, by
\cite{ShiJohanssonHong13TAC},
\cite{ShiProutiereJohansson15SICON},
\cite{YangMengShiHongJohansson16TAC}.
 Then $g_q:\reals^n\to\reals$ fitting Assumption \ref{time varying assumption} (c) is
$$g_q(x)=\frac{1}{4} \sum_{i,j=1}^k a_*(x_i-x_j)^2.$$
Bounds on $a_{ij}(t)$ ensure that Assumption \ref{time varying assumption} (d) holds. 

As in the autonomous case, $CS\subset A_q$ for each $q$ and thus $CS\subset A$.
If the union of $q$-th communication graphs between agents, over all $q\in Q$, is connected, and 
and $a_{ij}(q,t)=a_{ji}(q,t)>0$ for all $t\in[0,\infty)$ 
for every undirected edge between $i$-th and $j$-th agent in the $q$-th graph,
then $A\subset CS$ and thus $A=CS$. 
Under Assumption \ref{sigma assumption}, Theorem \ref{time varying theorem} 
implies that, in this case, agents reach consensus. 

Extending the arguments above to include local dynamics as in Example \ref{linear q local example} 
recovers \cite[Theorem 5.3]{ShiProutiereJohansson15SICON}. 
Considering local dynamics given by $\xdot_i=P_{D_i}(x_i)-x_i$, as at the end of Example \ref{linear q local example},
recovers \cite[Theorem 3.1]{ShiJohanssonHong13TAC}.
\end{example}

\subsection{Demipositive maximal monotone dynamics}

Demipositivity generalizes the property of the subdifferential of a convex function $f$ that
guarantees that solutions to $\xdot\in-\partial f$ converge to minima of $f$. Original definition
given by \cite{Bruck75JFA} was in an infinite-dimensional Hilbert space setting. Simplified to
the finite-dimensional case, the definition is this:
a maximal monotone mapping $M:\reals^n\rrarrows\reals^n$ is {\em demipositive} if
there exists $a\in M^{-1}(0)$ such that, if $v\cdot(x-a)=0$ for some $v\in M(x)$ then $x\in M^{-1}(0)$. 
For a proper, lsc, and convex $f$, $\partial f$ is demipositive. 
For a differentiable function $h(x,y)$ strictly convex in $x$ and strictly concave in $y$, the mapping
$(x,y)\mapsto\left(\nabla_x h(x,y),-\nabla_y h(x,y)\right)$ is demipositive, but
without strictness, this may fail: that same mapping for $h(x,y)=x^2+xy$ is not demipositive. 

Let $M:\reals^n\rrarrows\reals^n$ be a demipositive maximal monotone mapping with $M^{-1}(0)$ 
compact. Let $K\subset\reals^n$ be compact and such that $K\cap M^{-1}(0)=\emptyset$. Then
$$\inf_{a\in M^{-1}(0)} \inf_{x\in K} \inf_{v\in M(x)} v\cdot(x-a)$$
is positive and thus there exists a lsc function $W:\reals^n\to[0,\infty)$ with $W(x)=0$ if and only
if $x\in M^{-1}(0)$ and such that $\frac{d}{dt}\frac{1}{2}\|\sol(t)-a\|^2\leq-W(\sol(t))$
for every solution to $\xdot\in -M(x)$, every $a\in M^{-1}(0)$.  
Theorem \ref{uber theorem} can be then used to prove:

\begin{theorem} 
\label{demi theorem}
Let $Q={1,2,\dots,p}$. 
For each $q\in Q$, let $M_q:\reals^n\rrarrows\reals^n$ be a demipositive maximal monotone mapping.
Suppose that 
$$A:=\bigcap_{q\in Q} M^{-1}(0)\not=\emptyset.$$
Let $\sigma$ be a switching signal that satisfies Assumption \ref{sigma assumption} 
and $\sol:[0,\infty)\to\reals^n$ be a uniformly continuous solution to \eqref{ss}. 
Then $\sol(t)$ converges, as $t\to\infty$ to a point in $A$. 
\end{theorem}


\section{Acknowledgment}                               
The work by R. Goebel was partially supported by the Simons Foundation Grant 315326.
This work by R. Sanfelice was partially supported 
by the National Science Foundation under CAREER Grant no. ECS-1450484, Grant no. ECS-1710621, and Grant no. CNS-1544396, by the Air Force Office of Scientific Research under Grant no. FA9550-16-1-0015, and by the Air Force Research Laboratory under Grant no. FA9453-16-1-0053.

\end{document}